\documentclass[12pt]{amsart}

\usepackage{amscd,amssymb,amsopn,amsmath,amsthm,mathrsfs,graphics,amsfonts,enumerate,verbatim,calc,
}
\usepackage{mathrsfs}
\usepackage[all]{xy}
\usepackage[normalem]{ulem}
\newcommand\redout{\bgroup\markoverwith
{\textcolor{red}{\rule[.5ex]{2pt}{0.4pt}}}\ULon}
\usepackage{xcolor}
\usepackage[colorlinks=true,linkcolor=blue,citecolor=blue]{hyperref}
\usepackage{color}

\usepackage[OT2,OT1]{fontenc}
\newcommand\cyr{%
\renewcommand\rmdefault{wncyr}%
\renewcommand\sfdefault{wncyss}%
\renewcommand\encodingdefault{OT2}%
\normalfont
\selectfont}
\DeclareTextFontCommand{\textcyr}{\cyr}

\usepackage{amssymb,amsmath}

\usepackage[top=1in, bottom=1in, right=1in, left=1in]{geometry}
\newtheorem{theorem}{Theorem}[section]
\newtheorem{lemma}[theorem]{Lemma}
\newtheorem{proposition}[theorem]{Proposition}

\newtheorem{corollary}[theorem]{Corollary}

\newtheorem{question}[theorem]{Question}

\theoremstyle{definition}
\newtheorem{definition}[theorem]{Definition}
\newtheorem{remark}[theorem]{Remark}
\newtheorem{definitions}[theorem]{Definitions}
\theoremstyle{remark}

\newtheorem{example}[theorem]{Example}

\newcommand{\Ass}{\operatorname{Ass}}

\newcommand{\rad}{\operatorname{rad}}

\newcommand{\depth}{\operatorname{depth}}

\newcommand{\fm}{\frak{m}}
\newcommand{\fp}{\frak{p}}

\newcommand{\fn}{\frak{n}}

\dedicatory{Dedicated to the memory of Professor Shiro Goto}

\begin{document}

\title[Tight Hilbert function]{Tight Hilbert Polynomial and F-rational local rings}
\author{{Saipriya Dubey}, {Pham Hung Quy}
\and{Jugal Verma}}
\thanks{The first author is supported by the Senior Research Fellowship of HRDG, CSIR, Government of India. The second author is partially supported by a fund of Vietnam National Foundation for Science
and Technology Development (NAFOSTED) under grant number
101.04-2020.10.}
\thanks{Data sharing not applicable to this article as no datasets were generated or analysed during the current study.\\
  The authors have no conflicts of interest to declare. All co-authors have seen and agree with the contents of the manuscript and there is no financial interest to report. We certify that the submission is original work and is not under review at any other publication.}
\address{Saipriya Dubey \newline\indent Department of Mathematics, Indian Institute of Technology Bombay, \newline \indent 
Mumbai, India, Email- {\normalfont{sdubey@math.iitb.ac.in}} \vspace{0.28cm}
\newline\indent  Pham Hung Quy \newline\indent Department of Mathematics, FPT University,
\newline \indent
Hanoi, Vietnam, Email- {\normalfont {quyph@fe.edu.vn }}
\vspace{0.28cm}
\newline\indent  Jugal Verma \newline\indent Department of Mathematics, Indian Institute of Technology Bombay, \newline \indent 
Mumbai, India, Email- {\normalfont {jkv@iitb.ac.in}}}



\begin{abstract} Let $(R,\fm)$ be a Noetherian local ring of prime characteristic $p$ and $Q$ be an $\fm$-primary parameter ideal. We give criteria for F-rationality of $R$ using the tight Hilbert function $H^*_Q(n)=\ell(R/(Q^n)^*)$ and the coefficient $e_1^*(Q)$ of the  tight Hilbert polynomial 
$P^*_Q(n)=\sum_{i=0}^d(-1)^ie_i^*(Q)\binom{n+d-1-i}{d-i}.$
We obtain a lower bound for the tight Hilbert function of $Q$ for equidimensional excellent local rings that generalises a result of Goto and Nakamura. We show that if $\dim R=2 $, the Hochster-Huneke graph of $R$ is connected and this lower bound is achieved then $R$ is F-rational.
Craig Huneke asked if the $F$-rationality of unmixed local rings may be characterized by the vanishing of $e_1^*(Q).$ We construct examples to show that without additional conditions, this is not possible.
Let  $R$ be  an excellent, reduced, equidimensional Noetherian local ring and $Q$  be generated by parameter test elements.
 We find formulas for $e_1^*(Q), e_2^*(Q), \ldots, e_d^*(Q)$   in terms of Hilbert coefficients of $Q$,  lengths of local cohomology modules of $R,$ and the length of  the tight closure of the zero submodule of $H^d_{\fm}(R).$ Using these we prove:
$R$ is F-rational  $\iff e_1^*(Q)=e_1(Q) \iff \depth R\geq 2$ and $e_1^*(Q)=0.$
\end{abstract}

\maketitle

\thispagestyle{empty}
\section{Introduction}
The theory of tight closure created by Hochster and Huneke in the 1980's introduced several types of local rings such as F-regular, weakly F-regular, F-rational and F-injective local rings,  see for example \cite {HH1990},\cite{HH2},\cite{Sm97}, etc. It is well known that the Hilbert coefficients can be  used to characterize regular, Cohen-Macaulay and Buchsbaum  local rings. It is natural to expect that F-singularities could be characterized using a certain kind of Hilbert polynomial that involves the tight closure of ideals. The first step in this direction was taken by Shiro Goto and Y. Nakamura. In response to a conjecture of 
K. Watanabe and K. Yoshida \cite{WY}, Goto and Nakamura \cite{Got_Nak} 
proved the following interesting characterization of F-rational local rings. The length of an $R$-module $M$ is denoted by $\ell_R(M).$
The tight closure of an ideal $I$ is denoted by $I^*,$ see Section \ref{sec:prel} for definitions.
\begin{theorem} [\bf Goto-Nakamura, 2001] \label{gn} Suppose $R$ has prime characteristic and it is an equidimensional local ring of dimension $d.$  Suppose that $R$ is a homomorphic image of a Cohen-Macaulay local ring. Then 
{\rm (1)} $e_0(Q)\geq \ell_R(R/Q^*)$ for every $\fm$-primary parameter ideal $Q$ in $R.$\\
{\rm (2)} If $\dim R/\fp=d$ for all $\fp\in \Ass (R),$ and $e_0(Q)=\ell_R(R/Q^*)$ for some parameter ideal $Q$ in $R,$ then $R$ is a Cohen-Macaulay F-rational local ring.
\end{theorem}
For a recent treatment of Goto-Nakamura theorem, see \cite{colen_mult}.
Since $Q^*$ is contained in the integral closure $\overline{Q}$ of $Q,$ 
$e_0(Q)=e_0^*(Q).$ Therefore the   F-rationality  of $R$ is a consequence of the equality
$e^*_0(Q)=\ell(R/Q^*)$ for rings mentioned in  (2) above. This was an indication
that F-singularities could be characterized in terms of the \textit{tight Hilbert function} $H^*_Q(n)=\ell(R/(Q^n)^*).$  Let  $I$ be an $\fm$-primary ideal of $R$ and $R$ be  analytically unramified, i.e. the $\fm$-adic completion $\hat{R}$ is reduced. By a theorem of Rees \cite{rees}, $H^*_I(n)$ is given by a polynomial $P^*_I(n)$ for large $n.$ We call it the \textit{tight Hilbert polynomial of $I$} and write it as
$$P^*_I(n)=\sum_{i=0}^d(-1)^ie_i^*(I)\binom{n+d-1-i}{d-i}.$$
The coefficient $e_0^*(I)$ is the multiplicity $e_0(I)$ of $I.$ The other coefficients $e_i^*(I)\in \mathbb{Z}$ are called the \textit{tight Hilbert coefficients} of $I.$ 
 The tight Hilbert polynomial was introduced in \cite{GMV2019} where it was proved that an analytically unramified Cohen-Macaulay local ring $R$ having prime characteristic is F-rational if and only if $e_1^*(Q)=0$ for some  ideal $Q$ generated by a system of parameters of $R.$ 
This paper is motivated by the following question of  Craig Huneke 
\begin{question}\label{hun_que} Is it true that an unmixed Noetherian local ring $R$ is F-rational if and only if for some ideal $Q$ of $R$ generated by a system of parameters, $e_1^*(Q)=0$?
\end{question} 
We provide a negative answer to Question \ref{hun_que}, see Proposition \ref{prop:genofexampl}. We show that F-rationality can be characterized by the vanishing of  $e_1^*(Q)$ where $Q$ is an ideal generated by parameter test elements which form a system of parameters of $R$ where $R$ is reduced, excellent and equidimensional local Noetherian ring, see Corollary \ref{cor:main}.

This paper is organized as follows. In Section \ref{sec:prel}, we review the necessary background material related to tight closure of ideals, test ideals, F-rational local rings, excellent rings and the tight closure of the zero submodule of $H^d_\fm (R).$ 
In Section \ref{sec:F-rat}, we generalize the result of Goto-Nakamura [Theorem \ref{gn} (1)] for equidimensional excellent local rings by proving a lower bound for the tight Hilbert function.

\begin{theorem}\label{thm:ineq_tytHil_multipli_intr} Let $(R,\fm)$ be an equidimensional excellent local ring of prime characteristic $p$ and $Q$ be an ideal generated by a system of parameters for $R$. Then for all $n\ge 0,$
$$\ell (R/(Q^{n+1})^*) \ge \ell(R/Q^*) \binom{n+d}{d}.$$
\end{theorem}
\begin{corollary} 
Let $(R,\fm)$ be a reduced equidimensional excellent local ring of prime characteristic $p$ and $Q$ be an ideal generated by a system of parameters for $R$. Then $$e_0(Q)\geq \ell(R/Q^*).$$
\end{corollary}
In the next result, we show that if equality holds for some $n$ in Theorem \ref{thm:ineq_tytHil_multipli_intr} then $R$ is F-rational which can be considered as a generalization of Goto-Nakamura result [Theorem \ref{gn} (2)] under additional hypothesis.
\begin{theorem}  Let $(R, \fm)$ be a Noetherian local ring of dimension $d$ and prime characteristic $p.$ Let $(S,n)$ be a Cohen-Macaulay local ring of dimension $d$ and $Q(R)$ be the total quotient ring of $R$ such that $R\subseteq S \subseteq Q(R)$ and $S$ is a finite $R$-module. Let $Q$ be an ideal of $R$ generated by a system of parameters. Suppose that for some fixed $n\geq 0$,
$$\ell(R/(Q^{n+1})^*)=e_0(Q)\binom{n+d}{d}.$$
Then $R=S.$ In particular $R$ is F-rational.
\end{theorem}
If $d=2$ and the Hochster-Huneke graph of $R,$ denoted by $\mathcal{G}(R)$,  is connected then we can take $S$ in the above theorem to be the $S_2$-ification of $R$ and obtain the following
\begin{corollary}
Let $(R, \frak m)$ be a Noetherian local ring with $\dim (R/\fp)=2$ for all $\fp\in \Ass R$ of prime characteristic $p$ such that $\mathcal{G}(R)$ is connected. If for an ideal $Q$ generated by a system of parameters for $R$ and for some $n\ge 0,$ 
$$\ell (R/(Q^{n+1})^*) =e_0(Q) \binom{n+2}{2}$$
then $R$ is F-rational.
\end{corollary}

Let $(R,\fm)$ be a $d$-dimensional local Noetherian ring and $I$ be an $\fm$-primary ideal. {Then the \textit{Hilbert function} of $I$ is defined as $H_I(n)=\ell(R/{I^n}).$ For large $n,$ it coincides with a polynomial of degree $d$  called the \textit{Hilbert polynomial} of $I$} and it is written as 
 $$P_I(n)=e_0(I)\binom{n+d-1}{d}-e_1(I)\binom{n+d-2}{d-1}+\cdots +(-1)^de_d(I).$$
If  $R$ is analytically unramified then by a Theorem of Rees \cite{rees}, the \textit{normal Hilbert function} of an $\fm$-primary ideal $I,$ namely $\overline{H_I}(n)=\ell(R/\overline{I^n})$ coincides with a polynomial of degree $d$ for large $n.$ This polynomial is called the\textit{ normal Hilbert polynomial of $I$} and is given by
 $$\overline{P_I}(n)=e_0(I)\binom{n+d-1}{d}-\overline{e_1}(I)\binom{n+d-2}{d-1}+\cdots +(-1)^d\overline{e_d}(I).$$
\rm{In \cite{MTV1990} M. Moral\'es,  N. V. Trung and O. Villamayor  characterized regular local rings in terms of the equality $\overline{e_1}(Q) = e_1(Q)$ for a parameter ideal $Q$ of an excellent analytically unramified local ring. It is worth noting that this result was proved in \cite{MSV} by replacing the excellence hypothesis of $R$ with its unmixedness.
In Section \ref{sec:on equality Buchs}  we find an analogous characterization for F-rational local rings as a consequence of explicit formulas for the tight Hilbert coefficients in terms of the lengths of local cohomology modules $H^j_{\fm}(R)$ for $0\leq j \leq d-1,$ $e_i(Q)$ for $0\leq i \leq d$ and $\ell(0^*_{H^d_{\frak m}(R)}).$
\begin{theorem}
Let $(R,\fm)$ be an excellent reduced equidimensional local ring of prime characteristic $p$ and dimension $d\geq 2.$  Let $x_1,x_2,\ldots,x_d$ be parameter test elements and $Q=(x_1,x_2,\ldots,x_d)$ be $\fm$-primary. Then \\
{\rm (1) }$e_1^*(Q)=e_0(Q)-\ell(R/Q^*)+e_1(Q) \text{ and } e_j^*(Q)=e_j(Q)+e_{j-1}(Q) \text{ for all } 2\leq j \leq d,$\\
{\rm (2)} $\displaystyle e_1^*(Q)=\sum_{i=2}^{d-1} \binom{d-2}{i-2}\ell(H^i_{\frak m}(R)) + \ell(0^*_{H^d_{\frak m}(R)}),$\\
{\rm (3)} $\displaystyle e_i^*(Q)=(-1)^{i-1}\left[ \sum_{j=0}^{d-i}\binom{d-i-1}{j-2}\ell(H^j_{\fm}(R))+\ell(H^{d-i+1}_{\fm}(R))\right]  \text{ for } i=2, \ldots, d-1$ and\\
{\rm (4)} $e_d^*(Q)=(-1)^{d-1}\ell(H^1_{\fm}(R)).$
\end{theorem}
\begin{corollary}
Let $(R,\fm)$ be an excellent reduced equidimensional local ring of prime characteristic $p$ and dimension $d\geq 2.$  Let $x_1,x_2,\ldots,x_d$ be parameter test elements and $Q=(x_1,x_2,\ldots,x_d)$ be $\fm$-primary. Then the following are equivalent.\\
\centerline{{\rm (i)} $R $  is F-rational
{\rm (ii)}  $e_1^*(Q)=e_1(Q)$
\rm{(iii)} $e_1^*(Q)=0$ and $\depth R\geq 2.$}
\end{corollary}
In Section \ref{sec:hun_coun}, we construct examples to illustrate some of the above results.
\subsection{Notation and Conventions}
All the rings in this paper are commutative Noetherian rings with multiplicative identity 1. We use $(R,\fm, k)$ to denote local ring $R$ with unique maximal ideal $\fm$ and the residue field $k:=R/\fm.$  For basic results on Cohen-Macaulay rings, excellent rings, tight closure, Hilbert functions and multiplicity we refer the reader to \cite{BH} and \cite{Mat}.

{\bf Acknowledgements.} J. K. Verma would like to thank Prof. Craig Huneke for inviting him to University of Virginia in 2019 for discussions and for asking Question 1.2 which has led to this paper. Thanks are also due to Ian Aberbach for informing  us about his paper [1]. We thank the referees for a careful reading and several  suggestions which improved the paper.

\section{Preliminaries}\label{sec:prel} 
\noindent {In this section, we set up some notation and recall results needed in later sections.}
\subsection{Background on tight closure} Let $R$ be a commutative ring and $I$ be an ideal of $R.$ An element $x\in R$ is said to be \textit{integral} over $I$ if $$x^n+a_1x^{n-1}+a_2x^{n-2}+\cdots +a_n=0$$ for some $a_i\in I^i$ for $1\leq i \leq n.$ The \textit{integral closure} of $I,$ denoted by $\overline{I}$ is the collection of all elements that are integral over $I.$\\
Let $R$ be a Noetherian ring of prime characteristic $p$ and $R^\circ$ denote the subset of $R$ consisting of all elements which are not in any minimal prime ideal of $R.$ For $I=(x_1,\ldots ,x_n),$ let $I^{[p^e]}=(x_1^{p^e},\ldots ,x_n^{p^e}).$ The\textit{ tight closure} of $I,$ denoted by $I^*,$ is the set of all elements $x$ for which there exists some $c\in R^\circ$ such that $cx^{p^e}\in I^{[p^e]}$ for all $p^e>>0.$ An ideal $I$ is said to be \textit{tightly closed} if $I=I^*.$
For any ideal $I,$ we have $I\subseteq  I^*\subseteq \overline{I}.$
\begin{definition}
The \textit{test ideal} of $R,$ denoted by $\tau(R)$ is the ideal generated by elements $c\in R$ which satisfies  any of the following equivalent conditions.\\
{\rm (i)} $cx^q\in I^{[q]}$ for all $q=p^0,p^1,p^2,\ldots, $ whenever $x\in I^*$ for any ideal $I$ of $R.$\\
{\rm (ii)} $cx\in I$ whenever $x\in I^*$ for any ideal $I$ of $R.$\\
An element of $\tau(R) \cap R^\circ$ is called a \textit{test element}.
\end{definition}
 A Noetherian ring  $R$ is said to be \textit{weakly F-regular} if every ideal of $R$ is tightly closed.
Note that the test ideal of $R$ is the unit ideal if and only if $R$ is weakly F-regular.
Recall that a \textit{parameter ideal of height $n$} is an ideal of height $n$ generated by $n$ elements. For  excellent local equidimensional rings, parameter ideals are those generated by a part of a system of parameters for $R$ \cite{Sm_95}. 

\begin{definition}
The \textit{parameter test ideal} of $R,$ denoted by $\tau_{par}(R),$ is the ideal generated by $c\in R$ such that $cI^*\subset I$ for all parameter ideals $I$ of $R$ (equivalently, $cx^q\in I^{[q]}$ for all $q=p^e,$ $e=0,1,2,\ldots$). An element of $\tau_{par}(R)\cap R^\circ$ is called a \textit{parameter test element}.
\end{definition}
\begin{definition}
A Noetherian ring $R$ is called  \textit{F-rational} if all parameter ideals are tightly closed.
\end{definition}
Let $(R,\fm)$ be a $d$-dimensional local Noetherian ring and $x_1,\ldots ,x_d$ be a system of parameters. Then the local cohomology module $H^d_{\fm}(R)$ can be expressed as the $d^{th}$ cohomology of the \v Cech complex with respect to $x:=x_1,\ldots ,x_d$ since $H^d_{\fm}(R)\cong H^d_{I}(R),$ where $I=(x_1,\ldots ,x_d).$ Any element of $H^d_{\fm}(R)$ can be represented as $\eta:=\left[\frac{r}{x_1^ix_2^i\cdots x_d^i} \right].$
Let $R$ be a ring of characteristic $p>0.$ The Frobenius map $F:R\rightarrow R$ defined by $F(r)=r^p$ naturally induces an action called the Frobenius action on $H^d_{\fm}(R)$ which takes an element $\eta=\left[\frac{r}{(x_1x_2\cdots x_d)^i} \right]$ to $F(\eta) =\left[\frac{r^p}{(x_1x_2\cdots x_d)^{ip}} \right].$ Similarly, the e$^{th}$ iteration of the Frobenius map $F^e:R \rightarrow R$ defined as $F^e(r)=r^{p^e}$ induces a similar action on $H^d_{\fm}(R).$
\begin{definition} Let $(R,\fm)$ be a Noetherian local ring of characteristic $p.$ Then
$$0^*_{H^d_{\fm}(R)}=\{\eta\in H^d_{\fm}(R):\exists \; c\in R^{\circ}\text{ such that }\; cF^e(\eta)=0 \text{ for all }e>>0 \}.$$ 
\end{definition} 
We record a result from \cite{Sm94} which reveals the interplay of tight closure of the zero submodule of $H^d_{\fm}(R)$ with tight closure of ideal generated by a system of parameters of $R.$
\begin{theorem}\cite[Proposition 3.3(i)]{Sm94}\label{Smthm:0*_Frat}
Let $(R,\fm)$ be an excellent equidimensional local ring of dimension $d,$ and let $x_1,\ldots ,x_d$ be a system of parameters. Then any $z\in (x_1,\ldots ,x_d)^*$ uniquely determines an element $\eta=\left[\frac{z}{x_1x_2\cdots x_d} \right]\in 0^*_{H^d_{\fm}(R)}.$ Conversely, if $\eta=\left[\frac{z}{x_1x_2\cdots x_d} \right]\in 0^*_{H^d_{\fm}(R)},$ then $z\in (x_1,\ldots ,x_d)^*.$
\end{theorem} 
\begin{remark}\label{rem:tytcl=0iffRisFrat}
Note that if $R$ is Cohen-Macaulay, $\eta=\left[\frac{z}{x_1x_2\cdots x_d} \right]\in 0^*_{H^d_{\fm}(R)}$ and $\eta=0$  if and only if {$z\in (x_1,\ldots ,x_d)$}. Therefore Theorem \ref{Smthm:0*_Frat} implies that an excellent Cohen-Macaulay local ring $(R,\fm)$ of dimension $d$ is F-rational if and only if $0^*_{H^d_{\fm}(R)}=0.$  
\end{remark}
\subsection{Excellent Rings}
Very often, results in this paper and many results for tight closure assume that the given local ring is excellent. We shall use the following properties of  excellent rings frequently.\\
{\rm (1)} Let $(R,\fm)$ be an excellent local ring with $\fm$-adic completion $\hat{R}$ and $I$ be an $\fm$-primary ideal. Then $I^*\hat{R}=(I\hat{R})^*$ \cite[Proposition 10.3.18]{BH}.\\
{\rm (2)} Any excellent reduced local ring is analytically unramified \cite[Theorem 70]{Mat}.\\ 
{\rm (3)} Test elements exist in  reduced excellent local rings \cite[Theorem 6.1 (a)]{HH2}.\\  
{\rm (4)} If $R$ is excellent then it is a homomorphic image of Cohen-Macaulay ring \cite[Corollary 1.2]{Kaw2002}.}

\section{The tight Hilbert function and F-rationality of $R$}\label{sec:F-rat}
In this section, we give a generalization  of Goto-Nakamura results [Theorem \ref{gn}] for equidimensional excellent local rings. We provide a lower bound for tight Hilbert function and show that when the lower bound is achieved then the ring is F-rational under some additional conditions  on $R.$ Let us first prove a crucial lemma required for this purpose. Lemma \ref{tight intersec} follows from \cite[Theorem 8.20]{HH94}. However, we are giving a simpler proof of Lemma \ref{tight intersec}(b). We thank the referee or giving us a clear proof of the next lemma.
\begin{lemma}\label{tight intersec}
Let $(R,\fm)$ be an equidimensional excellent local ring of prime characteristic $p$ and $Q$ be an $\fm$-primary parameter ideal.\\
{\rm (a)} Then for all $n \ge 0$ we have $Q^n \cap (Q^{n+1})^* = Q^n Q^*$.\\
{\rm (b)} $Q^n/Q^n Q^*$ is a free $R/Q^*$-module of rank $\binom{n+d-1}{d-1}$, where $d = \dim R$.
\end{lemma}
\begin{proof}
{\rm (b)}	 We note that $Q^n$ is a $R$-module generated by monomials of degree $n$ in $x_1, \ldots, x_d$ which form minimal generators of $Q^n$ since $x_1,\ldots ,x_d$ are analytically independent \normalfont{\cite[Theorem 5]{NR}}.
	 Let $A = \mathbb F_p[x_1, \ldots, x_d]$ be the polynomial subring of $R$ generated by $x_1, \ldots, x_d$. Set $q = (x_1, \ldots, x_d)A$.  Let $m_1, \ldots, m_t$ be monomials in the $x_i$ of degree $n$ that form a minimal generating set of the finite $R/Q^*$-module $Q^n / Q^n Q^*$ (since any monomial of greater degree will sit in $Q^{n+1} \subseteq Q^n Q^*$). Suppose we have $u_i \in R$ such that $z=\sum_{i=1}^t u_i m_i \in Q^n Q^*$. To show that the $R/Q^*$-module $Q^n / Q^n Q^*$ is free, we must show that each $u_i \in Q^*$. For each $1\leq i \leq t$, set $J_i := (m_1, \ldots, \widehat{m_i}, \ldots, m_t)A$. Then since $Q^n Q^* \subseteq (Q^{n+1})^*$, we have $u_i m_i \in (Q^{n+1})^* + J_i R = (q^{n+1}R)^* + J_i R \subseteq ((q^{n+1} + J_i)R)^*$. Thus, $u_i \in ((q^{n+1}+J_i)R)^* :_R m_i \subseteq
(((q^{n+1}+J_i):_A m_i)R)^*$ by \cite[Theorem 2.3]{AHS95}. But it is easy to see in the polynomial ring $A$ that $(q^{n+1}+J_i):_A m_i \subseteq q$. Thus, $u_i \in (q R)^* = Q^*$.   
\end{proof}
\begin{theorem}\label{thm:ineq_tytHil_multipli} Let $(R,\fm)$ be an equidimensional excellent local ring of prime characteristic $p$ and $Q$ be an ideal generated by a system of parameters for $R$. Then for all $n\ge 0,$
$$\ell (R/(Q^{n+1})^*) \ge \ell(R/Q^*) \binom{n+d}{d}.$$
\end{theorem}
\begin{proof} We have 
$$\ell (R/(Q^{n+1})^*) = \sum_{k=0}^{n} \ell ((Q^k)^*/(Q^{k+1})^*).$$
For each $k$ we have
$$\ell \left(\frac{(Q^k)^*}{(Q^{k+1})^*} \right) \ge \ell \left(\frac{Q^k +(Q^{k+1})^*}{(Q^{k+1})^*}\right) = \ell \left(\frac{Q^k }{Q^k \cap (Q^{k+1})^*} \right) = \ell \left(\frac{Q^k }{Q^k Q^*} \right).$$
Since $Q^k$ is minimally generated over $R$ by ${k+d-1} \choose {d-1}$ generators, the base-changed module $Q^k / (Q^k Q^*)$ is also generated over $R/Q^*$ by ${k+d-1} \choose {d-1}$ generators.  As it must be free on these generators by Lemma \ref{tight intersec}, 
\[
    \ell((Q^k)^* / (Q^{k+1})^*) \geq \ell (Q^k / Q^k Q^*) = \ell(R/Q^*) {{k+d-1} \choose {d-1}}.
    \] 
Therefore 
$$\ell (R/(Q^{n+1})^*) \ge \ell(R/Q^*) \sum_{k=0}^{n} \binom{k+d-1}{d-1} = \ell(R/Q^*) \binom{n+d}{d}.$$
The proof is complete.
\end{proof}
\begin{corollary} \label{cor:conj}
Let $(R,\fm)$ be a reduced equidimensional excellent local ring of prime characteristic $p$ and $Q$ be an ideal generated by a system of parameters for $R$. Then $$e_0(Q)\geq \ell(R/Q^*).$$
\end{corollary}
\begin{proof}
Since $R$ is analytically unramified, by using Theorem \ref{thm:ineq_tytHil_multipli} for $n>>0$ we have,
$$\left[ e_0(Q)-\ell(R/Q^*) \right] \binom{n+d}{d}-e_1^*(Q)\binom{n+d-1}{d-1}+\cdots +(-1)^de_d^*(Q)\geq 0.$$ Therefore $e_0(Q)\geq \ell(R/Q^*).$
\end{proof}
The following lemma provides  equivalent conditions for F-rationality of Cohen-Macaulay rings. 
\begin{lemma}\label{lemma:TFAEtytcl}
 Let $(R,\fm)$ be a Cohen-Macaulay local ring of prime characteristic $p.$ Let $Q$ be an ideal of $R$ generated by a system of parameters.
Then the following are equivalent.\\
{\rm (a) } $Q^*=Q,$\\
{\rm (b) } $(Q^n)^*=Q^n$ for all $n\geq 1.$\\
{\rm (c) } $(Q^n)^*=Q^n$ for some $n\geq 1.$\\
\end{lemma}
\begin{proof}
(a) $ \implies $ (b). Observe that, using \cite[Proposition 4.2]{GMV2019}, $Q^n \cap (Q^{n+1})^*=Q^*Q^n$ for all $n\geq 1.$ Let $Q^*=Q.$ Apply induction on $n.$ The $n=1$ case is an assumption. Suppose that $(Q^n)^*=Q^n$  for $n=1, 2, \ldots, r.$ 
As $(Q^{r+1})^* \subset (Q^r)^*=Q^r,$ we have 
$$(Q^{r+1})^* = (Q^{r+1})^* \cap Q^r=Q^*Q^r=Q^{r+1}.$$
By induction $(Q^n)^*=Q^n$ for all $n\geq 1.$\\
(b) $ \implies $ (c). This is clear.\\
(c) $ \implies $ (a). Let  $(Q^n)^*=Q^n$ for some $n\geq 1.$ Therefore
$ Q^n = Q^{n-1}\cap (Q^n)^*=Q^*Q^{n-1}.$ Hence $Q^*\subseteq Q^{n}: Q^{n-1}=Q.$ Therefore $Q^*=Q.$
\end{proof}

\begin{theorem} Let $(R, \fm)$ be a Noetherian local ring of dimension $d$ and prime characteristic $p.$ Let $(S,n)$ be a Cohen-Macaulay local ring of dimension $d$ and $Q(R)$ be the total quotient ring of $R$ such that $R\subseteq S \subseteq Q(R)$ and $S$ is a finite $R$-module. Let $Q$ be an ideal of $R$ generated by a system of parameters. Suppose that for some fixed $n\geq 0$, $$\ell(R/(Q^{n+1})^*)=e_0(Q)\binom{n+d}{d}.$$
Then $R=S.$ In particular $R$ is F-rational.
\end{theorem}
\begin{proof}
Using \cite[Proposition 10.1.5]{BH} we get $(Q^nS)^*\cap R \subseteq (Q^n)^*.$ Let $f=[S/\fn: R/\fm]. $ Then we obtain the following

\begin{equation} \label{dqv_eqnTHF1}
\ell_R(R/(Q^{n+1})^*)  \leq \ell_R(R/(Q^nS)^*\cap R) \leq  \ell_R(S/(Q^{n+1}S)^*)  
 \leq\ell_R(S/Q^{n+1}S),
 \end{equation}
 
 \begin{equation}\label{dqv_eqnTHF2}
\ell_R(S/Q^{n+1}S) =f \ell_S(S/(Q^{n+1}S))=fe_0(QS)\binom{n+d}{d}=e_0(Q)\binom{n+d}{d}.
\end{equation}

Therefore, if $\ell(R/(Q^{n+1})^*)=e_0(Q)\binom{n+d}{d},$ then $(Q^{n+1}S)^*=(Q^{n+1}S).$ As $S$ is Cohen-Macaulay, using Lemma \ref{lemma:TFAEtytcl} it follows that $(QS)^*=QS$ and therefore $S$ is F-rational. Now consider the exact sequence of finite $R$-modules
$$0 \to R\to S \to C\to 0,$$
where $C=S/R.$ From \eqref{dqv_eqnTHF1} and \eqref{dqv_eqnTHF2}, it follows that $(Q^{n+1})^*=(Q^{n+1}S)^*\cap R=Q^{n+1}S\cap R.$ Tensor this sequence with $R/Q^{n+1}$ to get the exact sequence of $R$-modules
$$0 \to R/(Q^{n+1})^*\to S/Q^{n+1}S\to C/Q^{n+1}C\to 0.$$ As $\ell(R/(Q^{n+1})^*)=e_0(Q)\binom{n+d}{d},$ using \eqref{dqv_eqnTHF1} and \eqref{dqv_eqnTHF2}, we get $ \ell_R( R/(Q^{n+1})^*)=\ell_R(S/Q^{n+1}S)$ which yields $C=Q^{n+1}C.$ By Nakayama's lemma, $C=0.$ This means $R=S.$ In particular $R$ is F-rational.
\end{proof}
We discuss a  relationship of $e_1^*(Q)$ with $S_2$-ification. Let $(R, \fm, k)$ be a Noetherian local ring of dimension $d.$
We recall a few facts about $S_2$-ification of $R$ from \cite{HH1994}.
\begin{definitions}
{\rm (1)} We say that $R$ is \textit{equidimensional} if $\dim R/\fp=d$ for all minimal primes $\fp$ of $R.$ If $R$ is equidimensional and it  has no embedded associated primes, then  $R$ is called \textit{unmixed}. 

{\rm (2)} Let $(R, \frak m)$ be an equidimensional local ring of dimension $d$. The \textit{Hochster-Huneke graph} $\mathcal{G}(R)$ is a graph where the vertices are the minimal prime ideals of $R$ and the edges are the pairs of prime ideals $(P_1, P_2)$ with $\mathrm{ht}(P_1 + P_2) = 1$.

{\rm (3)} Let $(R,\fm,k)$ be an equidimensional and unmixed local ring. We say that a ring $S$ is an \textit{$S_2$-ification} of $R$ if \\
{\rm (i)} $S$ lies between $R$ and its total quotient ring,\\
{\rm (ii)} $S$ is module-finite over $R$ and is $S_2$ as an $R$-module, and \\
{\rm (iii)} for every element $s\in S\setminus R,$ the ideal $D(s):=\{r\in R:rs\in R\}$ has height at least two.
\end{definitions}
If $R$ is $S_2$ then $\mathcal{G}(R)$ is connected. Moreover $\mathcal{G}(R)$ is connected if and only if the $S_2$-ification of $R$ is local \cite[Theorem 3.6]{HH1994}.
\begin{corollary}
Let $(R, \frak m)$ be a Noetherian local ring with $\dim (R/\fp)=2$ for all $\fp\in \Ass R$ of prime characteristic $p$ such that $\mathcal{G}(R)$ is connected. If for an ideal $Q$ generated by a system of parameters for $R$ and for some $n\ge 0,$ 
$$\ell (R/(Q^{n+1})^*) =e_0(Q) \binom{n+2}{2}$$
then $R$ is F-rational.
\end{corollary}
\begin{proof} By the result above, the $S_2$-ification $S$ of $R$ is a Cohen-Macaulay local ring that is a finite $R$-module.
\end{proof}

\section{On the equality $e_1^*(Q)=e_1(Q)$ and F-rational local rings}\label{sec:on equality Buchs} 
In \cite{MTV1990} M. Moral\'es,  N. V. Trung and O. Villamayor proved the following characterization of regular local rings.
\begin{theorem}\cite[Theorem 1,2]{MTV1990} Let $(R,\fm)$ be an analytically unramified excellent local domain and $I$ be an $\fm$-primary parameter ideal. If $\overline{e}_1(I) = e_1(I)$ then $R$ is a regular and $\overline{I^n}=I^n$ for all $n.$
\end{theorem}
In this section, we find explicit formulas for the tight Hilbert coefficients of  an ideal $Q$ generated by system of parameters that are parameter test elements, in terms of the lengths of local cohomology modules $H^j_{\fm}(R)$ for $0\leq j \leq d-1,$ $e_i(Q)$ for $0\leq i \leq d$ and $\ell(0^*_{H^d_{\frak m}(R)}).$ We use these formulas to  characterize F-rationality of the ring in terms of the equality $e_1^*(Q)=e_1(Q)$ and also in terms of vanishing of $e_1^*(Q)$ under the condition that $\depth R\geq 2.$\\
Let $(R,\fm)$ be a local ring of dimension $d$ and $I$ be any $\fm$-primary parameter ideal of $R.$ It is well known that $\ell(R/I)\ge e_0(I).$ Moreover, $R$ is Cohen-Macaulay if and only if $\ell(R/I)=e_0(I)$ for some (and hence for all) $I.$ Recall that $R$ is called \textit{Buchsbaum} if $\ell(R/I)-e_0(I)$ is independent of the choice of $I.$ 
\begin{definition} Let $(R,\fm)$ be a $d$-dimensional Noetherian local ring. An $\fm$-primary parameter ideal $I$ is said to be \textit{standard} if 
$$\ell(R/I)-e_0(I)=\sum_{i=0}^{d-1}\binom{d-1}{i}\ell(H^i_{\fm}(R)).$$  
\end{definition}

The  following result  due to  Linquan Ma and Pham Hung Quy plays a crucial role for proving  a characterization of F-rationality in terms of vanishing of $e_1^*(Q)$ for $\fm$-primary parameter ideals generated by parameter test elements.
\begin{theorem}\cite[Theorem 4.3]{MQ2021} \label{thm:tytcland0*}
Let $(R, \frak m)$ be an excellent equidimensional local ring such that $\tau_{par}(R)$ is $\fm$-primary. Let $Q$ be an ideal generated by a system of parameters contained in $\tau_{par}(R).$ Then we have
$$\ell(Q^*/Q) = \sum_{i=0}^{d-1} \binom{d}{i}\ell(H^i_{\frak m}(R)) + \ell(0^*_{H^d_{\frak m}(R)}).$$
\end{theorem} 
\begin{remark}\label{rem:main}
(i) If $Q$ is an ideal generated by a system of parameters of $R$ consisting of parameter test elements then it is a standard system of parameters of $R$ \cite[Remark 5.11]{H1998} and \cite[Proposition 3.8]{Sch_1983}.\\
(ii) If $Q$ is generated by a standard system of parameters, then the Hilbert polynomial, infact Hilbert function of $Q$ can be found in \cite[Corollary 3.2]{Sch_1979}, \cite[Corollary 4.2]{T86}, \cite[Theorem 7]{GoMV}, etc. For $n\geq 0,$ 
$$\ell(R/Q^n)=\sum_{i=0}^d(-1)^ie_i(Q)\binom{n+d-1-i}{d-i}, \text{ where }$$\\
$$e_i(Q)=(-1)^i\sum_{j=0}^{d-i}\binom{d-i-1}{j-1} \ell(H^j_{\fm}(R) )\mbox{ for all } i=1,2,\ldots, d.$$\\
(iii) If $x_1,\ldots ,x_d\in \tau_{par}(R)$ and $Q=(x_1,\ldots ,x_d)$ is $\fm$-primary in $(R,\fm)$ then $Q\subseteq \tau_{par}(R)$ and taking radicals on both sides, we obtain $\fm\subseteq \rad(\tau_{par}(R))$ which implies that $\tau_{par}(R)$ is either $\fm$-primary or $R.$
\end{remark}
\begin{theorem}\label{tyt_coeff}
Let $(R,\fm)$ be an excellent reduced equidimensional local ring of prime characteristic $p$ and dimension $d\geq 2.$  Let $x_1,x_2,\ldots,x_d$ be parameter test elements and $Q=(x_1,x_2,\ldots,x_d)$ be $\fm$-primary. Then \\
{\rm (1)} $e_1^*(Q)=e_0(Q)-\ell(R/Q^*)+e_1(Q) \text{ and } e_j^*(Q)=e_j(Q)+e_{j-1}(Q) \text{ for all } 2\leq j \leq d,$\\
{\rm (2)} $e_1^*(Q)=\sum_{i=2}^{d-1} \binom{d-2}{i-2}\ell(H^i_{\frak m}(R)) + \ell(0^*_{H^d_{\frak m}(R)}),$\\
{\rm (3)} $e_i^*(Q)=(-1)^{i-1}\left[ \sum_{j=0}^{d-i}\binom{d-i-1}{j-2}\ell(H^j_{\fm}(R))+\ell(H^{d-i+1}_{\fm}(R))\right]  \text{ for } i=2, \ldots, d.$
\end{theorem}
\begin{proof}
(1) By Lemma \ref{tight intersec}, $Q^n/Q^nQ^*$ is a free $R/Q^*$-module of rank $\binom{n+d-1}{d-1}$ for all $n\geq 1$ and  by \cite[Lemma 3.1]{Ab96}, $(Q^{n+1})^*=Q^nQ^*$ for all $n\geq 1.$ Hence 
$$\ell (Q^n/Q^nQ^*)= \ell (Q^n/(Q^{n+1})^*)= \ell (R/Q^*) \binom{n+d-1}{d-1} .$$
Thus $\ell(R/(Q^{n+1})^*)=\ell(R/Q^n)+\ell(R/Q^*)\binom{n+d-1}{d-1}$ for all $n\geq 1.$ By Remark \ref{rem:main}(ii) the tight Hilbert function of $Q$ is given by 
\begin{align*}
H_Q^*(n) & =e_0(Q) \binom{n+d-2}{d}-e_1(Q)\binom{n+d-3}{d-1}+\cdots+(-1)^de_d(Q)+\ell(R/Q^*) \binom{n+d-2}{d-1}\\
&=\sum_{i=0}^de_i(Q)(-1)^i\binom{n+d-2-i}{d-i}+\ell(R/Q^*) \binom{n+d-2}{d-1}\\
&=\sum_{i=0}^de_i(Q)(-1)^i\left[ \binom{n+d-1-i}{d-i}- \binom{n+d-2-i}{d-1-i} \right]+\ell(R/Q^*) \binom{n+d-2}{d-1}\\
&=e_0(Q)\binom{n+d-1}{d}-[e_0(Q)-\ell(R/Q^*)+e_1(Q)]\binom{n+d-2}{d-1}\\
&+\sum_{i=2}^d(-1)^i[e_i(Q)+e_{i-1}(Q)]\binom{n+d-i-1}{d-i}.
\end{align*}
Equating like terms on both sides, we obtain the desired formulas.\\
(2) From (1) we have $e_1^*(Q)=e_0(Q)-\ell(R/Q^*)+e_1(Q).$ On the other hand, since $Q$ is standard, using Remark \ref{rem:main}(iii) and Theorem \ref{thm:tytcland0*} we have
\begin{eqnarray*}
\ell(R/Q^*) &=& \ell(R/Q) - \sum_{i=0}^{d-1} \binom{d}{i}\ell(H^i_{\frak m}(R)) - \ell(0^*_{H^d_{\frak m}(R)})\\
&=& e_0(Q) + \sum_{i=0}^{d-1} \binom{d-1}{i}\ell(H^i_{\frak m}(R)) - \sum_{i=0}^{d-1} \binom{d}{i}\ell(H^i_{\frak m}(R)) - \ell(0^*_{H^d_{\frak m}(R)})\\
&=& e_0(Q) - \sum_{i=1}^{d-1} \binom{d-1}{i-1}\ell(H^i_{\frak m}(R)) - \ell(0^*_{H^d_{\frak m}(R)}),
\end{eqnarray*}
where the second equality above follows from Remark \ref{rem:main}(i). Hence 
\begin{equation}
e_1^*(Q) = \sum_{i=1}^{d-1} \binom{d-1}{i-1}\ell(H^i_{\frak m}(R)) + \ell(0^*_{H^d_{\frak m}(R)}) + e_1(Q).
\end{equation}
Furthermore by Remark \ref{rem:main}(ii), it follows that 
\begin{align*}
e_1^*(Q) &= \sum_{i=1}^{d-1} \binom{d-1}{i-1}\ell(H^i_{\frak m}(R)) + \ell(0^*_{H^d_{\frak m}(R)}) -\sum_{j=0}^{d-1}\binom{d-2}{j-1}\ell(H^j_{\frak m}(R))\\
&=\sum_{i=1}^{d-1}\left[ \binom{d-1}{i-1}-\binom{d-2}{i-1}\right]\ell(H^i_{\frak m}(R)) + \ell(0^*_{H^d_{\frak m}(R)})\\
&=\sum_{i=1}^{d-1} \binom{d-2}{i-2}\ell(H^i_{\frak m}(R)) + \ell(0^*_{H^d_{\frak m}(R)})\\
&=\sum_{i=2}^{d-1} \binom{d-2}{i-2}\ell(H^i_{\frak m}(R)) + \ell(0^*_{H^d_{\frak m}(R)}).
\end{align*}
(3) Using Remark \ref{rem:main}(i)-(ii), we obtain
\begin{align*}
\ell(R/Q^n)&=\sum_{i=0}^d\binom{n+d-1-i}{d-i}(-1)^i e_i(Q) \mbox{ for all } n\geq 1 ,\\
(-1)^ie_i(Q)&=\sum_{j=0}^{d-i}\binom{d-i-1}{j-1} \ell(H^j_{\fm}(R) )\mbox{ for all } i=1,2,\ldots, d,\\
\ell(R/Q)-e_0(Q)&=\sum_{j=0}^{d-1}\binom{d-1}{j}\ell(H_{\fm}^j(R))\\
e_d(Q)&=(-1)^d\ell(H_{\fm}^0(R)).
\end{align*}
In the formulas above, we follow the convention $\binom{n}{-1}=1$ if $n=-1$ and $\binom{n}{-1}=0$ if $n\neq -1.$
 By the above formulas and the fact that $R$ is reduced and equidimensional, 
 $$e_d^*(Q)=e_d(Q)+e_{d-1}(Q)=(-1)^{d-1}\ell(H^1_{\fm}(R)).$$  
 Next we find the formulas for $e_i^*(Q)$ where $i=2,3, \ldots, d$ in terms of the lengths of the local cohomology modules. Put
 $h^j=\ell(H_{\fm}^j(R)).$
 \begin{eqnarray*}
 e_i^*(Q)&=& e_i(Q)+e_{i-1}(Q)\\
 &=& (-1)^i\sum_{j=0}^{d-i}\binom{d-i-1}{j-1}h^j+(-1)^{i-1}\left[
 \sum_{j=0}^{d-i}\binom{d-i}{j-1}h^j+h^{d-i+1}\right]\\
 &=&(-1)^{i-1}\left[ \sum_{j=0}^{d-i}\binom{d-i-1}{j-2}h^j+h^{d-i+1}\right].
 \end{eqnarray*}
\end{proof}
In the $\dim 1$ case Question \ref{hun_que} has an affirmative answer. Let $(R,\fm)$ be a $1$-dimensional analytically unramified local ring and $I=(a)$ be $\fm$-primary. Since $R$ is reduced and $\dim R=1,$ $R$ is Cohen-Macaulay. Let $$P_{I}^*(n)=e(I)n-e_1^*(I).$$ If $e_1^*(I)=0$ then $R$ is F-rational.  Let $(b)\subseteq \fm$ be a minimal reduction of $\fm.$ By Brian\c{c}on-Skoda Theorem, $\overline{(b)}=(b)^*.$ As $R$ is F-rational, $(b)^*=(b).$ Thus $(b)=\overline{(b)}=\fm.$ Hence $R$ is a regular local ring. In the case $\dim R\geq 2,$ we have answered Huneke's question with some additional hypothesis which can be derived as a consequence of Theorem \ref{tyt_coeff}.
\begin{corollary}\label{cor:main}
Let $(R,\fm)$ be an excellent reduced equidimensional local ring of prime characteristic $p$ and dimension $d\geq 2.$  Let $x_1,x_2,\ldots,x_d$ be parameter test elements and $Q=(x_1,x_2,\ldots,x_d)$ be $\fm$-primary. Then the following are equivalent.\\
{\rm (i)} $R$   is F-rational. \\
{\rm (ii)}  $e_1^*(Q)=e_1(Q).$\\
{\rm (iii)} $e_1^*(Q)=0$ and   $ \depth R\geq 2.$
\end{corollary}

\begin{proof} (i) $\iff$ (ii): If $R$ is F-rational then $R$ is Cohen-Macaulay.  Therefore  $Q^n=(Q^{n})^*$ for all  $n\geq 1$ \cite[Corollary 4.3]{GMV2019}. Hence {$\ell(R/(Q^{n+1})^*)=e_0(Q)\binom{n+d}{d}$ for all $n\geq 0$} which implies that $e_1^*(Q)=e_1(Q)=0.$ 

Conversely, let $e_1^*(Q)=e_1(Q).$ Using Theorem \ref{tyt_coeff}(1),  $e_0(Q)=\ell(R/Q^*).$ As $R$ is unmixed, by \cite{Got_Nak}, $R$ is F-rational.\\
(i) $\iff$ (iii): If $R$ is F-rational then it is Cohen-Macaulay so that (iii) holds. Conversely, let $e_1^*(Q)=0$ and $\depth R\geq 2.$ By Theorem \ref{tyt_coeff}(2), it follows that $0^*_{H^d_{\frak m}(R)}=0$ and $H^i_{\frak m}(R)=0$ for $2\leq i\leq d-1.$ As $\depth R\geq 2,$ $H^0_{\frak m}(R)=H^1_{\frak m}(R)=0.$ Hence $R$ is Cohen-Macaulay ring with $0^*_{H^d_{\frak m}(R)}=0.$ By Remark \ref{rem:tytcl=0iffRisFrat}, it follows that $R$ is F-rational.
\end{proof}

\section{A Counterexample to Huneke's question}\label{sec:hun_coun}
\normalfont{ We provide a negative answer to Huneke's question by constructing examples of unmixed local rings in which $e_1^*(Q)=0$ for an ideal $Q$ generated by a system of parameters  but $R$ is not F-rational.  
} The next proposition gives a class of examples where $0^*_{H^d_{\fm}(R)}$ vanishes.
\begin{proposition}\label{prop:exampl}
Let $(R, \frak m)$ be an equidimensional reduced local ring of dimension $d$, and $\mathrm{Ass}R = \{P_1, P_2\}$. Suppose $R/P_1$ and $R/P_2$ are both F-rational and $\dim R/(P_1 + P_2) \le d-2$. Then $0^*_{H^d_{\frak m}(R)}=0.$
\end{proposition}
\begin{proof}
Consider the long exact sequence of local cohomology arising from the following short exact sequence.
$$0 \to R \to R/P_1\oplus R/P_2 \to R/(P_1+P_2)\to 0.$$
Since $\dim R/(P_1 + P_2) \le d-2,$ it follows that $H^{i}_{\frak m}(R/(P_1 + P_2))=0$ for $i=d-1,d.$ This implies that $H_{\frak m}^d(R) \cong H_{\frak m}^d(R/P_1) \oplus H_{\frak m}^d(R/P_2).$ Clearly, $0^*_{H_{\frak m}^d(R)}\cong 0^*_{H_{\frak m}^d(R/P_1)}\oplus 0^*_{H_{\frak m}^d(R/P_2)}.$ Since $R/P_i$ is F-rational for $i=1,2$ we have $0^*_{H_{\frak m}^d(R/P_i)}=0$ which implies that $0^*_{H_{\frak m}^d(R)}=0.$
\end{proof}
\begin{lemma}\label{lemma:rel_of_multi}
Let $(R, \frak m)$ be an equidimensional reduced local ring of dimension $d$, and $\mathrm{Ass}\;R = \{P_1, P_2\}$. Then for any $\fm$-primary parameter ideal $Q$ in $R,$ $$e_0(Q)=e_0\big((Q+P_1)/P_1\big)+e_0\big((Q+P_2)/P_2\big).$$
\end{lemma}
\begin{proof}
Since $R$ is reduced, $\ell_{R_{P_i}}(R_{P_i})=1$ for $i=1,2.$
By the associativity formula for multiplicity we get, 
\begin{align*}
e_0(Q)&=e_0\big((Q+P_1)/P_1\big)\ell(R_{P_1})+e_0(Q+P_2/P_2)\ell(R_{P_2})\\
&=e_0\big((Q+P_1)/P_1\big)+e_0\big((Q+P_2)/P_2\big).
\end{align*}
\end{proof}
\begin{proposition}\label{prop:genofexampl}
Let $(R, \frak m)$ be an equidimensional reduced local ring of dimension $d$ and prime characteristic $p$ with $\mathrm{Ass}\;R = \{P_1, P_2\}$. Suppose $R/P_1$ and $R/P_2$ are both F-rational and $\dim R/(P_1 + P_2) \le d-2$. Then $R$ is not Cohen-Macaulay and for any ideal generated by a system of parameters $Q,$ we have $e_1^*(Q) = 0$. 
\end{proposition}
\begin{proof}
Since $R/P_i$ is F-rational we have $(Q^{n+1} R/P_i)^* = (Q^{n+1} + P_i)/P_i$ for $i=1,2.$ Using \cite[Proposition 6.25(a)]{HH1990}, we have $(Q^{n+1})^* + P_i = Q^{n+1} + P_i$ for all $i = 1, 2$. Thus $(Q^{n+1})^* \subseteq (Q^{n+1} + P_1) \cap (Q^{n+1} + P_2)$. Moreover, $x \in (Q^{n+1})^*$ if and only if the image of $x$ in $R/P_i$ is contained in $(Q^{n+1} R/P_i)^* = (Q^{n+1} + P_i)/P_i$ for $i = 1, 2$. Hence $(Q^{n+1})^* = (Q^{n+1} + P_1) \cap (Q^{n+1} + P_2)$. Therefore we have the short exact sequence
$$0 \to R/(Q^{n+1})^* \to R/(Q^{n+1} + P_1) \oplus R/(Q^{n+1} + P_2) \to R/(Q^{n+1} + P_1 + P_2) \to 0$$
for all $n \ge 0$. Thus we have 
\begin{align*}
\ell(R/(Q^{n+1})^*) &=\ell\big(R/(Q^{n+1} + P_1)\big)+\ell\big( R/(Q^{n+1} + P_2)\big)- \ell\big(R/(Q^{n+1} + P_1 + P_2)\big)\\
&=\left[e_0\big((Q+P_1)/P_1\big)+e_0\big((Q+P_2)/P_2\big)\right]\binom{n+d}{d}- \ell\big(R/(Q^{n+1} + P_1 + P_2)\big)\\
&= e_0(Q) \binom{n+d}{d} - \ell(R/(Q^{n+1} + P_1 + P_2)),
\end{align*}
where the last equality follows from Lemma \ref{lemma:rel_of_multi}.

Since $\ell(R/(Q^{n+1} + P_1 + P_2))$ is a polynomial of degree atmost $ d-2,$ $e_1^*(Q) = 0$ for all $Q$. 
Consider the the short exact sequence of $R$-modules
$$0 \to R \to R/P_1 \oplus R/P_2 \to R/(P_1 + P_2) \to 0.$$ Since $\depth (R/P_1 \oplus R/P_2)=d>\dim (R/(P_1 + P_2)),$ by the depth Lemma $\depth R\leq d-1.$ Hence $R$ is not Cohen-Macaulay.
\end{proof}
We construct an example to  show that the condition $\depth R \geq 2$ in Corollary \ref{cor:main} is not superfluous for characterization of F-rationality in terms of vanishing of $e_1^*(Q).$
\begin{example}
Let $S=\mathbb{F}_p[|X,Y,Z,W|]$ and $R=\frac{S}{I\cap J},$ where $I=(X,Y)$ and $J=(Z,W).$ Let the lower case letters denote images of the upper case letters. Put $\frak m=(x,y,z,w).$ Let $a=x+z,\, b=y+w.$ Then $a,b$ is a system of parameters. Set $Q=(a,b).$ Since $R$ is Buchsbaum $$\ell \left(\frac{R}{Q} \right)-e_0(Q)=\sum_{i=0}^{d-1}\binom{d-1}{i}\ell(H_{\frak m}^i(R))=\ell(H_{\frak m}^1(R))=1,$$ Note that $H_{\frak m}^1(R)\cong H_{\frak m}^0(R/\frak m)\cong  R/\frak m.$
Using $e_i(Q)=(-1)^i\sum_{j=0}^{d-i}\binom{d-i-1}{j-1}\ell(H_{\frak m}^j(R)),$ we get $e_1(Q)=-\ell(H_{\frak m}^1(R))=-1,$ $e_2(Q)=0.$ Since $R$ is Buchsbaum and $0^*_{H^d_{\frak{m}}}(R)=0,$ it follows that $\tau_{par}(R)=\frak{m}.$ Thus by Theorem \ref{tyt_coeff}(1), $e_2^*(Q)=e_2(Q)+e_1(Q)=-1$ and $e_1^*(Q)=0$ by Proposition \ref{prop:genofexampl}. Therefore $$P_{Q}^*(n)=2\binom{n+1}{2}-1.$$
\end{example}

\begin{example} We construct a complete local domain of dimension $2$ that is not F-rational but there exists an ideal $Q$ generated by a system of parameters $Q$ for which $e_1^*(Q)=0.$ Let $k$ be a field of prime characteristic $p\geq 3$ and $R =k[[x^4, x^3y, xy^3, y^4]]$. We have the $S_2$-ification of $R$ is the local ring $S = k[[x^4, x^3y,x^2y^2, xy^3, y^4]]$. We have $C:=S/R \cong k$, so that $\ell(C/JC)=1$ for any $\fm$-primary ideal $J$ of $R.$ Let $Q$ be any $\fm$-primary ideal parameter ideal of $R$.
Consider the short exact sequence,
$$0 \to R/(Q^{n+1})^*\to S/(Q^{n+1}S)^* \to C \to 0.$$ We have $$\ell(R/(Q^{n+1})^*)=\ell(S/(Q^{n+1})^*S)-1.$$ Since
$S$ is F-regular,
$$\ell (R/(Q^{n+1})^*) = e_0(Q) \binom{n+2}{2} - 1$$ 
for all $n\ge 1$. Since $S/\fn \cong R/\fm, $ $e_0(Q)=e_0(QS).$  Hence $e_1^*(Q) = 0$.
\end{example}

\end{document}